\documentclass[12pt]{amsart}
\usepackage{amssymb}
\usepackage{amsmath}
\usepackage{bbm}
\usepackage{color}
\usepackage{verbatim}

\definecolor{darkred}{rgb}{0.75,0,0.3}

\newcommand\AND{\quad\text{and}\quad}

\newcommand\C{\mathbb C}

\newcommand\dt{\mathsf{det}}

\newcommand\dps{\displaystyle}
\newcommand\Fb{\mathbf{F}}
\newcommand\Gb{\mathbf{G}}

\newcommand\kr{\mathsf{ker}}
\newcommand\la{\lambda}

\newcommand\N{\mathbb N}
\newcommand\Prob{\mathbb{P}}

\newcommand\R{\mathbb R}
\newcommand\res{\mathsf{res}}
\newcommand\spec{\mathsf{spec}}

\newcommand\uno{\mathbf{1}}

\numberwithin{equation}{section}

\newtheoremstyle{mythm}
  {9pt}
  {9pt}
  {\itshape}
  {0pt}
  {\bfseries}
  {}
  { }
  {\thmnumber{(#2)}\thmname{ #1}\thmnote{ #3}}

\newtheoremstyle{mydef}
  {9pt}
  {9pt}
  {\normalfont}
  {0pt}
  {\bfseries}
  {}
  { }
  {\thmnumber{(#2)}\thmname{ #1}\thmnote{ #3}}

\theoremstyle{mythm}
\newtheorem{thm}[equation]{Theorem.}
\newtheorem{pro}[equation]{Proposition.}
\newtheorem{lem}[equation]{Lemma.}
\newtheorem{cor}[equation]{Corollary.}

\theoremstyle{mydef}

\newtheorem{exa}[equation]{Example.}

\begin{document}$\,$ \vspace{-1truecm}
\title{Polyharmonic functions for finite graphs and Markov chains}
\author{\bf Thomas Hirschler, Wolfgang Woess}
\address{\parbox{.8\linewidth}{Institut f\"ur Diskrete Mathematik,\\ 
Technische Universit\"at Graz,\\
Steyrergasse 30, A-8010 Graz, Austria\\}}
\email{thirschler@tugraz.at, woess@tugraz.at}
\date{\today} 
\begin{abstract}
On a finite graph with a chosen partition of the vertex set into interior and boundary vertices, 
a $\la$-polyharmonic function is a complex function $f$ on the vertex set which satisfies
$(\la \cdot I - P)^n f(x) = 0$ at each interior vertex. Here, $P$ may be the normalised 
adjaceny matrix, but more generally, we consider the transition matrix $P$ of an arbitrary
Markov chain to which the (oriented) graph structure is adapted. After describing these
``global'' polyharmonic functions, we turn to solving the \emph{Riquier problem,} where
$n$ boundary functions are preassigned and a corresponding ``tower'' of $n$ successive
Dirichlet type problems are solved. The resulting unique solution will be polyharmonic
only at those points which have distance at least $n$ from the boundary. Finally, we compare
these results with those concerning infinite trees with the end boundary, as studied
by Cohen, Colonnna, Gowrisankaran and Singman, and more recently, by Picardello and Woess.
\end{abstract}

\maketitle

\markboth{{\sf T. Hirschler and W. Woess}}
{{\sf Polyharmonic functions}}
\baselineskip 15pt

\section{Introduction}\label{sec:intro}

In the setting of the classical Laplacian $\Delta$ on
a Euclidean domain, or the Laplace-Beltrami operator on
a Riemannian manifold, a polyharmonic function $f$ is one
for which $\Delta^n f = 0$. Their study goes back to work in the $19^{\text{th}}$ century, see e.g. 
{\sc Almansi}~\cite{Al}. A basic reference is the monograph by {\sc Aronszajn, 
Creese and Lipkin}~\cite{ACL}. A more recent one is the volume by 
{\sc Gazzola, Grunau and Sweers}~\cite{GGS}, with a nice introduction 
to classical problems from elasticity  where polyharmonic (in fact biharmonic) functions and 
$\Delta^2$ come up.

While there is a huge  body of literature in the smooth case, the literature in
the discrete setting is quite restricted: an early reference is {\sc Voronkova}~\cite{Vo},
who analysed the discretized version of $\Delta^2 f = 0$ in a half-strip 
$[0\,,\infty] \times [0\,,\,H]$. Other quite early references are 
{\sc Yamasaki}~\cite{Ya} and {\sc Kayano and Yamasaki}~\cite{KY}
who investigated the Green kernel for the bi-Laplacian on an infinite network, 
and a follow-up of this is {\sc Venkataraman}~\cite{Ve}. Biharmonic Laplacians on trees
where also studied by {\sc Cohen, Colonna and Singman}~\cite{CCS1}, \cite{CCS2}, 
seemingly without link to \cite{Ya} and \cite{KY}. Prior to that, 
{\sc Cohen, Colonna, Gowrisankaran and Singman}~\cite{CCGS}
were the first to undertake a detailed study of polyharmonic functions on infinite, 
locally finite trees.
In particular, for the standard Laplacian arising from simple random walk on a regular tree,
they provided a boundary integral representation which is an analogue of Almansi's 
expansion of polyharmonic functions on the unit disk. (To get a flavour of the many 
close analogies between the potential theory of the unit disk and regular trees, 
the reader is invited to
the introductory sections of {\sc Boiko and Woess}~\cite{BW}.)
Recently, {\sc Picardello and Woess}~\cite{PW} extended the study of \cite{CCGS} and 
proved, among outher, a boundary integral representation of $\lambda$-polyharmonic 
functions (see below for more details)  for arbitrary nearest neighbour transition 
operators on countable trees, not necessarily required to be locally finite.  

In all this work, \emph{finite} graphs, resp. Markov chains had only marginal appearances:
in \cite{Ya} for the biharmonic Green function of finite subnetworks of an infinite network,
and in \cite{CCGS} for finite trees and an associated boundary value problem for biharmonic
functions.  {\sc Anandam}~\cite{An} also studies polyharmonic functions on finite subtrees
of infinite trees.

In the present note, we elaborate a detailed account of the general 
finite case, in which the mentioned potential theoretic questions turn into issues 
of linear algebra which can be solved rather easily.

\medskip

\textbf{The setting.} We start with a finite set $X$, subdivided into the disjoint union of 
two non-empty subsets $X^o$, the \emph{interior,} and $\partial X$, the \emph{boundary}.
On $X$, we consider a stochastic \emph{transition matrix} $P = \bigl( p(x,y) \bigr)_{x,y \in X}$ 
with the following
properties, where $p^{(n)}(x,y)$ denotes the $(x,y)$-entry of the matrix power $P^n$.
\begin{enumerate}
 \item[(i)] For all $x \in X^o$, there is $w \in \partial X$ such that $p^{(n)}(x,w) > 0$ for
some $n$.  
 \item[(ii)] For all $w \in \partial X$, we have $p(w,w) = 1$, and thus $p(w,x)=0$ for 
all $x \in X \setminus\{w\}$.
 \item[(iii)] For all $w \in \partial X$, there is $x \in X^o$ such that $p^{(n)}(x,w) > 0$ for
some $n$.  
\end{enumerate}

Thus, $X$ can be given the structure of a digraph, where we have an oriented edge $x \to y$ when
$p(x,y) > 0$. Then (i) means that the boundary can be reached from any interior point by
an oriented path, (ii) means that each boundary point is absorbing, i.e., the only outgoing
edge is a loop at that point, and (iii) means that every boundary point is active in the sense
that it is reached by some oriented path from an interior point. 
In probabilistic terms, we have a Markov chain (random process) on $X$, whose evolution is governed
by $P\,$: if the current position is $x$ then the next step is from $x$ to $y$ with probability
$p(x,y)$.

\begin{exa}\label{exa:network}
The most typical situation is the one where we start with a finite \emph{resisitive network}, that is, 
a connected, non-oriented  graph $(X,E)$ where each edge $e= [x,y] = [y,x]$ carries a positive
\emph{conductance} $a(e) = a(x,y)$. Then we choose our partition $X = X^o \cup \partial X$,
and we set $m(x) = \sum_y a(x,y)$. The transition probabilities become 
$p(x,y) = a(x,y)/m(x)$, if $x \in X^o$ and $y \in X$, while $p(w,w) = 1$ for $w \in \partial X$.
This defines a reversible Markov chain which is absorbed in $\partial X$, 
see e.g. {\sc Woess}~\cite[Ch. 4]{WMarkov}. In particular, setting all $a(x,y)$ equal to $1$,
the conductances correspond to the adjacency matrix. 
\end{exa}

The transition matrix $P$ acts on functions (column vectors) $f: X \to \C$ by
$$
Pf(x) = \sum_y p(x,y)f(y)\,,
$$
and the (normalised) graph Laplacian is $I-P$, where $I= I_X$ is the identity matrix over $X$.
It is typically defined on $X$ without assigning a boundary $\partial X$, but the study undertaken
here makes sense only in presence of absorbing points. 
Note that the more direct analogue of the (negative definite) smooth Laplacian would in reality
be $P-I$. 
More generally, we shall work with suitable variant of $\lambda \cdot I -P$ for 
$\la \in \C$.

A \emph{$\lambda$-harmonic function} $h: X \to \C$ is one for which 
\begin{equation}\label{eq:laharm}
Ph(x) = \lambda \,h(x) \quad \text{for every}\quad x \in X^o\,.
\end{equation}
When $\la =1$, we speak of a harmonic function. When speaking of 
$\lambda$-\emph{polyharmonic functions} of order $n$,
we have two possible approaches: one is to look for functions $f : X \to \C$ which satisfy
\begin{equation}\label{eq:global-poly}
(\lambda \cdot I - P)^n f = 0 \quad \text{on }\; X.
\end{equation}
These \emph{global} $\lambda$-polyharmonic functions can be easily described.

The more interesting version is related with the pre-assignment of boundary values.
Let $P_{X^o}$ and $Q$ be the restrictions of $P$
to $X^o \times X^o$ and $X^o \times \partial X$, respectively. Then we define the $\la$-Laplacian
as the matrix given in block-form by
\begin{equation}\label{eq:Lapl}
\Delta_{\lambda} = \begin{pmatrix}  \lambda\cdot I_{X^o} - P_{X^o} & \; - Q \\[3pt]
                                                                0 & 0
                   \end{pmatrix} = 
                   \begin{pmatrix} \lambda\cdot I_{X^o} & 0 \\[3pt]
                                                      0 & I_{\partial X}
                   \end{pmatrix} - P
\,,
\end{equation}   
where the $0$s stand for the zero matrices in the respective dimensions. Here, the 
identity matrix over $\partial X$ is \emph{not} multiplied by $\lambda$, so that functions
annihilated by $\Delta_{\lambda}$ are $\lambda$-harmonic only in $X^o$. 

Our main focus is on polyharmonic functions in the sense that they satisfy
\begin{equation}\label{eq:poly}
\Delta_{\lambda}^n f = 0 
\end{equation}
on $X$, or -- more reasonably, as we shall see -- on the ``$n$-th interior'' of
$X$, i.e., all points in $X^o$ from which $\partial X$ cannot be reached in less
then $n$ steps. 
When $\la = 1$, the two notions \eqref{eq:global-poly} and \eqref{eq:poly} coincide. 

This note is organized as follows. In \S \ref{sec:global}, we first consider ordinary
harmonic and polyharmonic functions, that is the case $\la =1$. After recalling the
well known solution of the \emph{Dirichlet problem} for harmonic functions with
preassigned boundary values (Lemma \ref{lem:Dir}), we explain why all global harmonic 
functions in the sense of \eqref{eq:global-poly} (with $\la=1$) are indeed harmonic 
(Proposition \ref{pro:freeharm}). Then we look at all global $\la$-polyharmonic 
functions as in \eqref{eq:global-poly}. In this case, $\la$ must belong to the spectrum
of $P_{X^o}$, and the solutions can be described in terms of a Jordan basis (Proposition
\ref{pro:global-lambda}).

In \S \ref{sec:local}, we turn to studying $\Delta_{\la}$ and its powers, for $\la$
in the resolvent set of $P_{X^o}$ (the spectrum being settled in \S 2). There is a 
direct analogue to the solution of the Dirichlet problem, and again, any function 
which satisfies $\Delta_{\la}^n f = 0$ on all of $X$ must be $\la$-harmonic
(Proposition \ref{pro:lafreeharm}).
Finally, we give the precise formulation of the \emph{Riquier} problem, 
which consists in assigning boundary functions $g_1\,,...,g_n$ and -- loosely spoken --
searching for a function $f$ such that the boundary values of $\Delta_{\la}^{r-1}f$
coincide with $g_r$ for $r=1, \dots, n$. That problem for the special case of
finite trees is briefly touched in \cite{CCGS}. Here, we provide the general solution
(Theorem \ref{thm:Riquier}). 

Finally, in \S \ref{sec:compare}, we undertake a comparison of those
results with the case of infinite trees without leaves, which was studied
recently in \cite{PW} by use of Martin boundary theory.

All results of this note are achieved by applying basic tools from Linear Algebra
in the right way. We believe that this material provides a useful basis, firstly as a
link to the classical, smooth case (regarding the Laplacian on bounded domains),
and secondly, as a basis for handling and understanding polyharmonic functions
not only on infinite trees, but also on more general infinite graphs and their
boundaries at infinity.

\section{The Dirichlet problem, and global $\la$-polyharmonic functions}\label{sec:global}

We start with some observations on the case $\la= n = 1$, that is, ordinary harmonic functions.
We start with a simple observation on $\spec(P_{X^o})$, the set of eigenvalues of $P_{X^o}$.

\begin{lem}\label{lem:spec}
The spectral radius 
$\rho = \rho(P_{X^o}) = \max \{ |\la| : \la \in \spec(P_{X^o})\}$
satisfies $\rho < 1$.
\end{lem}

\begin{proof}[Proof (outline)]
Condition (i) on $P$ implies that for each $x \in X^o$, there is $n$ such that
$\sum_{v \in X^o} p^{(n)}(x,v) < 1$, that is, $P_{X^o}^n$ is strictly 
substochastic in the row of $x$. One easily deduces that there is $m$ such
that  $P_{X^o}^m$ is strictly substochastic in every row, which yields the claim. 
\end{proof}

The following \emph{solution of the Dirichlet problem} is folklore in the Markov chain 
community; see e.g. \cite[\S 6.A]{WMarkov}. It keeps being ``rediscovered'' by analysts 
who deviate into the discrete world, see for example {\sc Kiselman} \cite{Ki}. 

\begin{lem}\label{lem:Dir} 
 For every function $g: \partial X \to \C$ there is a unique harmonic function $h$ on $X$ such
that $h|_{\partial X} = g$. It is given by
$$
h(x) = \sum_{w\in \partial X} F(x,w) g(w)\,,
$$
where $F(x,w)$ is the probability that the Markov chain starting at $x$ hits $\partial X$ in
the point~$v$.
\end{lem}

We next want to describe the kernel $F(x,w)$ in matrix terminology. 
Let $\res(P_{X^o})= \C \setminus \spec(P_{X^o})$ be the resolvent set
of $P_{X^o}$. For $\la \in \res(P_{X^o})$, the \emph{resolvent} is the $X^o \times X^o$-matrix
\begin{equation}\label{eq:Green}
\Gb(\la) = \bigl( G(x,y|\la) \bigr)_{x,y \in X^o} = (\la \cdot I_{X^o} - P_{X^o})^{-1} 
\end{equation}
The kernels $G(x,y|\la)$ are called \emph{Green functions.} They are rational functions of
$\la$. Now we define the $X^o \times \partial X$-matrix
\begin{equation}\label{eq:F}
\Fb(\la) = \bigl( F(x,w|\la) \bigr)_{x \in X^o, w \in \partial X } = \Gb(\la) \,Q. 
\end{equation}
We can extend it to $X \times \partial X$ by setting $F(v,w|\la) = \delta_w(v)$ for 
$v,w \in \partial X$.
When $\la =1$, we just write $G(x,y)$  for $G(x,y|1)$ and $F(x,w)$ for $F(x,w|1)$. 
For $|\la| > \rho$, we can expand
$$
G(x,y|\la) = \sum_{n=0}^{\infty} p^{(n)}(x,y)/\la^{n+1} \AND
F(x,w|\la) = \sum_{n=0}^{\infty} f^{(n)}(x,w)/\la^n \,,
$$ 
where the probabilistic meaning is that for the Markov chain starting
at $x$, the probability to be at $y$ at time $n$ is $p^{(n)}(x,y)$,
while $f^{(n)}(x,w)$ is the probability that the first visit in $w \in \partial X$ 
occurs at time $n$.

Coming back to the Dirichlet problem, it is a straightforward
matrix computation to see that the function $h$, as defined in 
Lemma \ref{lem:Dir}, is harmonic. Its uniqueness follows from 
invertibility of $(I_{X^o} - P_{X^o})$. Instead, it may also be 
instructive to deduce uniqueness from the potential theoretic 
\emph{maximum principle:} every real valued harmonic function attains
its maximum on $\partial X$, see \cite[\S 6.A]{WMarkov}.

This also yields one way to see that the Markov chain must hit the boundary
almost surely, that is
$$
\sum_{w \in \partial X} F(x,w) = 1 \quad \text{for every }\;x\in X^o. 
$$
Namely, the unique harmonic extension of the constant boundary function $g \equiv 1$
is the constant function $h \equiv 1$ on $X$. Also, the function $x \mapsto F(x,w)$
provides the unique harmonic extension of the boundary function $g = \uno_v\,$.  

\begin{cor}\label{cor:mult}
The geometric and the algebraic multiplicity of the eigenvalue $\lambda = 1$ of $P$
coincide and are equal to $|\partial X|$. 
\end{cor}

\begin{proof} Lemma \ref{lem:Dir} yields that the geometric multiplicity is $|\partial X|$.
The characteristic polynomial of the matrix $P$ is 
$$
\chi_P(\la) = \dt (\lambda \cdot I - P) = (\lambda -1)^{|\partial X|} \chi_{P^o}(\la).
$$
By Lemma \ref{lem:spec}, $\chi_{P^o}(1) \ne 0$.
\end{proof}

Now we can easily describe all \emph{free} polyharmonic functions of order $n \ge 1$,
that is, those which satisfy $(I-P)^n f = 0$ on $X$.

\begin{pro}\label{pro:freeharm}
A function $f: X \to \C$ satisfies  $(I-P)^n f = 0$ if and only if $f$ is harmonic.
\end{pro}

\begin{proof}
Suppose $n\ge 2$, and let $h = (I-P)^{n-1}f$. Then $h$ is harmonic, and $(I-P)f = h$. Since
$(I-P)^{n-1}f = 0$ on $\partial X$, the function $h$ solves the Dirichlet problem
with boundary values $0$. Therefore $h = 0$, that is,  $(I-P)^{n-1}f = 0$. Proceeding
by induction, we obtain that $f$ is harmonic.
\end{proof}

Similarly, we can handle the case $(\la \cdot I -P)^n f = 0$, when $\la \ne 1$. First of all,
when $n \ge 2$ then the function $h= (\la \cdot I - P)^{n-1} f$ satisfies $Ph = \la \cdot h$. 
Second,  we see that $f = 0$ in $\partial X$, so (by abuse of notation) we consider $f$ as
a function on $X^o$. In other words, $\la \in \spec(P_{X^o})$. 

Let $\kappa= \kappa(\la)$ and 
$\mu = \mu(\la)$ be the algebraic and geometric eigenvalue multiplicities of $\la$. 
Let $h_1\,,\dots\,,h_{\mu}$ be a basis of $\kr(\la \cdot I_{X^o} - P_{X^o})$. 
For each $j \in \{1, \dots, \mu\}$, let $\kappa_j$ be the length of the associated
Jordan chain (= dimension of the associated Jordan block in the Jordan normal form).
That is, $\kappa_1 + \dots + \kappa_{\mu} = \kappa$, and 
we have functions $f_j^{(k)}$, $k=1, \dots, \kappa_j$ such that
$f_j^{(1)} = h_j$ and $(\la \cdot I_{X^o} - P_{X^o})f_j^{(k)} = f_j^{(k-1)}$ for 
$k \ge 2$. All those functions are extended to $X$ by assigning value $0$ on $\partial X$.
Then it is clear that 
$\;
\{ f_j^{(k)} : k=1, \dots, \kappa_j\,,\; j=1,\dots, \mu \}
\;$
is a basis of the linear space of all global $\la$-polyharmonic functions (of arbitrary order).
We subsume.

\begin{pro}\label{pro:global-lambda}
With the above notation, for $\la \in \spec(P_{X^o})$, the space of functions 
$f : X \to \C$ with $(\la \cdot I - P)^n f=0$ is spanned by 
$$
\bigl\{ f_j^{(k)} : k=1, \dots, \min\{n,\kappa_j\}\,,\; j=1,\dots, \mu \bigr\}.
$$
\end{pro}

\begin{cor}\label{cor:network}
For a finite network with boundary as in Example \eqref{exa:network}, every global
$\la$-polyharmonic $h$ function satisfies $Ph =\la \cdot h$, and $\la \in \spec(P) \subset \R$.
Furthermore, $h$ vanishes on $\partial X$ when $\la \ne 1$. 
\end{cor}

\begin{proof} If we define the diagonal matrix 
$M = \mathsf{diag}\bigl( \sqrt{m(x)}\, \bigr)_{x \in X^o}$,
then $M\, P_{X^o}\, M^{-1}$ is symmetric, so that the spectrum is real and the geometric and algebraic
multiplicities of the eigenvalues of $P_{X^o}$ coincide.
\end{proof}

\section{Boundary value problems for $\la$-polyharmonic functions}\label{sec:local}

In this section, we assume that $\la \in \res(P_{X^o})$ and study the operator (resp. matrix)
$\Delta_{\la}$ of \eqref{eq:Lapl} and its powers. 

\emph{Notation:} in accordance with the block 
form used above, for any function $f: X \to \C$ we write 
$f=\Bigl( {\dps f^o \atop \dps f^{\partial} } \Bigr)$, where $f^o = f|_{X^o}$ and 
$f^{\partial} = f|_{\partial X}\,$. Also, we write $\Delta_{\la}^o$ for the restriction of
the matrix of \eqref{eq:Lapl} to $X \times X^o$, that is, $\Delta_{\la}^o f = (\Delta_{\la} f)^o$. 

First of all, there is an obvious $\la$-variant of the solution of the Dirichlet problem.

\begin{lem}\label{lem:laDir} Let $\la \in \res(P_{X_o})$.
 For every function $g: \partial X \to \C$ there is a unique $\la$-harmonic function $h$ on $X$ such
that $h|_{\partial X} = g$. It is given by
$$
h(x) = \sum_{w\in \partial X} F(x,w|\la) g(w)\,, \quad x \in X^o\,, 
$$
where $F(x,w|\la)$ is defined by \eqref{eq:F}.
\end{lem}

\begin{proof} We write 
$h=\Bigl( {\dps h^o \atop \dps g} \Bigr)$,
where $h^o= h|_{X^o}$ and $g$ is the given boundary function. Then the equation
$\Delta_{\la} = 0$ transforms into  
$$
(\la \cdot I_{X^o} - P_{X^o})h^o = Qg\,,
$$
which has the unique solution $h^o = \Gb(\la)Qg\,$, as proposed.
\end{proof}

Next, we note that
$$
\Delta_{\la}^n = \begin{pmatrix}  (\lambda\cdot I_{X^o} - P_{X^o})^n&  
                                       \;-(\lambda\cdot I_{X^o} - P_{X^o})^{n-1} Q \\[3pt]
                                                                0 & 0
                   \end{pmatrix}.
$$
Thus, if wo look for a solution of $\Delta_{\lambda}^n h = 0$ then with 
$h=\Bigl( {\dps h^o \atop \dps g} \Bigr)$ as above, we get the equation
$$
(\la \cdot I_{X^o} - P_{X^o})^n h^o = (\lambda\cdot I_{X^o} - P_{X^o})^{n-1}Qg\,,
$$
which has the same solution as in Lemma \ref{lem:laDir}. Thus, we have the following
general version of Proposition \ref{pro:freeharm}.

\begin{pro}\label{pro:lafreeharm}
A function $f: X \to \C$ satisfies  $\Delta_{\la}^n f = 0$ on all of $X$ 
if and only if $f$ is $\la$-harmonic.
\end{pro}

For $n \ge 2$, what is more interesting is to assign further boundary conditions.
Recall that $\Delta_{\la}f$ always vanishes on $\partial X$. 
The analogue of the Dirichlet problem is the \emph{Riquier problem} of order $n$. We
assign $n$ boundary functions $g_1\,,\dots, g_n : \partial X \to \C$ and look
for a function $f: X \to \C$ such that we have a ``tower'' of boundary value problems for
functions $f_n\,,f_{n-1}\,, \dots, f_1=f : X \to \C$ as
follows: 
\begin{equation}\label{eq:Riquier} 
f_r  ={
\Bigl( {\dps f_r^o \atop \dps g_r} \Bigr)}\,,\quad \Delta_{\la}^o f_n = 0\,, \AND
\Delta_{\la}^o f_r = f_{r+1}^o\quad \text{for}\quad r= n\!-\!1, n\!-\!2, \dots, 1\,.   
\end{equation}

\begin{thm}\label{thm:Riquier} 
For $\la \in \res(P_{X^o})\,$, the unique solution $f = f_1$ of \eqref{eq:Riquier} is given by
$$
f(x) = \sum_{r=1}^n \bigl[\Gb(\la)^r \, Q\,  g_r\bigr](x) \,,\; x \in X^o\,,
$$ 
where $\Gb(\la)^r$ is the $r$-th matrix power of $\Gb(\la)$.
\end{thm}

\begin{proof} We use induction on $n$. For $n=1$, this is Lemma \ref{lem:laDir}.
Suppose the statement is true for $n-1$. The function $f_2$ is the solution of
the Riquier problem of order $n-1$ for the boundary functions $g_2\,,\dots, g_n$.
By the induction hypothesis, 
$$
f_2(x) = \sum_{r=2}^n \bigl[\Gb(\la)^{r-1}  \, Q \,g_r\bigr](x) \,,\; x \in X^o\,,
$$ 
and this is the unique solution. The last one of the ``tower'' of equations 
\eqref{eq:Riquier} is
$$
\Delta_{\la}^o f = f_2^o\,, \quad \text{where}\quad f  =
{\textstyle \Bigl( {\dps f^o \atop \dps g_1} \Bigr)}\,
$$ 
This can be rewritten as 
$$
(\la \cdot I_{X^o} - P_{X^o})f^o - Q\,g_1 = f_2^o\,. 
$$
Inserting the solution for $f_2^o$ and multiplying by $\Gb(\la)$, we
get the solution for $f$, and it is unique.
\end{proof}

Note that the solution $f$ does not satisfy \eqref{eq:poly} on all of $X^o$. This is
due to the fact that our discrete Laplacian is not infinitesmial. Let
\begin{equation}\label{eq:nbdry}
\partial^n X = \{ x \in X : p^{(k)}(x,w) > 0 \; \text{for some} \; w \in \partial X
\; \text{and} \; k \le n-1\}\,,
\end{equation}
the set of all points in $X$ from which $\partial X$ can be reached in $n-1$ or less
steps. Then $\Delta_{\la}^n f = 0$ only on the \emph{$n$-th interior} $X \setminus \partial^n X$,
while the values on $\partial^n X$ depend on the boundary functions $g_1\,,\dots, g_n\,$.
 
The functions $\la \mapsto G(x,y|\la)$ are rational, and the union of the set of their
poles is $\spec(P_{X^o})$. For $\la \in \res(P_{X^o})$, we can differentiate the identity\hfill
$
\la \cdot \Gb(\la) - P\,\Gb(\la) = I_{X^o}
$\\
$k$ times, and Leibniz' rule yields 
$$
(\la \cdot I_{X^o} - P_{X^o})\,\Gb^{(r)}(\la) = -k \cdot \Gb^{(r-1)}(\la)\,,
$$
where $\Gb^{(r)}(\la)$ is the (elementwise) $r$-th derivative of $\Gb(\la)$ with respect to
$\la$. From this, we get recursively for the matrix powers of $\Gb(\la)$
\begin{equation}\label{eq:deriveG}
\Gb(\la)^r = \frac{(-1)^{r-1}}{(r-1)!}\, \Gb^{(r-1)}(\la)\,.
\end{equation}
We can insert this in the formula of Theorem \ref{thm:Riquier} 
for an alternative form of the solution of the Riquier problem.

\section{Comparison with the case of infinite trees; examples}\label{sec:compare}

We now want to relate the preceding material, and in particular Theorem \ref{thm:Riquier}, 
with the potential theory of countable Markov chains, and more specifically, with Martin boundary theory
and $\la$-polyharmonic functions on trees, as studied in \cite{PW}. We choose and fix
a reference point (origin) $o \in X^o$ and consider the rational functions $\la \mapsto F(o,w|\la)$
of \eqref{eq:F}  for $\la \in \res(P_{X^o})$ and $w \in \partial X$. 
They have (at most) finitely many zeros.
Let 
$$
\res^*(P_{X^o}) = \res(P_{X^o}) \setminus 
\{ \la : F(o,w|\la) = 0 \;\text{for some} \; w \in \partial X\}\,.
$$
Every positive real $\la > \rho(P)$ belongs to $\res^*(P_{X^o})$, in particular,
$\la =1$. For $\la \in \res^*(P_{X^o})$, we define the \emph{$\la$-Martin kernel}
\begin{equation}\label{eq:K}
K^{(X)}(x,w|\la) = \frac{F(x,w|\la)}{F(o,w|\la)}\,, \quad x \in X\,,\; w \in \partial X\,.
\end{equation}
The function $x \mapsto K^{(X)}(x,w|\la)$ is the unique solution of the $\la$-Dirichlet
problem of Lemma \ref{lem:laDir} with value $1$ at the root $o$ and the boundary function 
$g_v$ proportional to $\delta_w\,$, that is
$g_w(v) = \delta_w(v)/F(o,w|\la)$. Thus, for a  generic boundary function 
$g :  \partial X \to \C$, we can write the solution of the $\la$-Dirichlet
problem for $x \in X^o$ as
\begin{equation}\label{eq:DirK}
\begin{aligned} h(x) &= \sum_{w\in \partial X} K^{(X)}(x,w|\la) \nu(w) 
=: \int_{\partial X} K^{(X)}(x,\cdot\,|\la)\,d\nu \;
\quad (x \in X^o), 
\; \text{ where}\\ \nu(w) &=  g(w)\, F(o,w|\la).
\end{aligned}
\end{equation}
The integral notation indicates that we think of $\nu=\nu_g$ as a complex distribution 
on $\partial X$. In the same way, the solution of the Riquier problem in Theorem 
\ref{thm:Riquier} can be written as
\begin{equation}\label{eq:RiqK}
\begin{aligned}
&f(x) = \sum_{r=1}^n \int_{\partial X}
K^{(X)}_r(x,\cdot\,|\la) \,  d\nu_r\,, \quad \text{where for }\;
w \in \partial X\\
&K^{(X)}_r(\cdot,w|\la) = \Gb(\la)^{r-1} K(\cdot,w|\la)
\AND \nu_r(w) = g_r(w)\, F(o,w|\la)\,.
\end{aligned}
\end{equation}

Now let us look at the case of a nearest neighbour transition operator $P = P_T$ on a countable
tree $T$ without leaves (i.e., vertices 
distinct from $o$ which have just one neighbour): there, the geometric boundary is attachted to the tree ``at infinity'',
and there is no ``interior'' of $T$ which appears as a subset of the vertex set: the interior
is $T$ itself. The Martin kernel $K^{(T)}(x,\xi|\la)$ is defined for $x \in T$ and $\xi \in \partial T$,
and it satisfies $(\la \cdot I - P) K^{(T)}(\cdot,\xi|\la) = 0$, without any restriction to a sub-matrix
such as $P_{X^o}\,$. In this setting, \cite[Thm. 5.4]{PW} says that any $\la$-polyharmonic
function $f$ of order $n$ on $T$ has a unique representation of the form
\begin{equation}\label{eq:polyT}
\begin{aligned}
f(x) &= \sum_{r=1}^n \int_{\partial T} K^{(T)}_r(x,\cdot\,|\la) \,  d\nu_r\,, 
\qquad \text{where} \\
K^{(T)}_r(x,\xi|\la) &= \frac{(-1)^{r-1}}{(r-1)!} \, \frac{d^{r-1}}{d\la^{r-1}} K(x,\xi|\la)
\qquad (x \in T,\; \xi \in \partial T),
\end{aligned}
\end{equation}
and $\nu_1\,,\dots, \nu_n$ are distributions on $\partial T$. 
The normalization is  slightly different here from the one chosen in \cite{PW}, and in particular,
\begin{equation}\label{eq:nrm}
(\la \cdot I_T - P_T)K^{(T)}_r(\cdot,\xi|\la) = K^{(T)}_{r-1}(\cdot,\xi|\la) \quad 
\text{for }\; r \ge 2\,. 
\end{equation}
 Let us compare the kernels $K^{(X)}_r$ and $K^{(T)}_r$. We have
\begin{equation}\label{eq:compare}
\begin{aligned}
(\la \cdot I_{X^o} - P_{X^o})^{r-1} K^{(X)}_r (\cdot,w|\la) &= K^{(X)} (\cdot,w|\la) 
\quad \text{for} \quad w \in \partial X\,,
\AND\\ 
(\la \cdot I_{T} - P_{T})^{r-1} K^{(T)}_r (\cdot,\xi|\la) &= K^{(T)} (\cdot,\xi|\la)
 \quad \text{for} \quad \xi \in \partial T\,.
\end{aligned}
\end{equation}
The only, but crucial difference is that in the first of the two identities, we
may multiply from the left by $\Gb^{(X)}(\la)^{r-1} = (\la \cdot I_{X^o} - P_{X^o})^{-(r-1)}$.
In the second identity, we may \emph{not} multiply by $\Gb^{(T)}(\la)^{r-1}$, where
$\Gb^{(T)}(\la) = (\la \cdot I_{T} - P_{T})^{-1}$ is the resolvent of $P$ as an operator
on the Hilbert space $\ell^2(T,m)$, with the weights $m(x)$ analogous to Example \ref{exa:network}
above. Indeed, $K^{(T)} (\cdot,\xi|\la)$ does in general not belong to $\ell^2(T,m)$. 

\medskip

\textbf{``Forward only'' Laplacians on finite and infinite trees}

We now consider a class of examples which constitute the finite analogue of \cite[\S 6]{PW}.
They were also studied, from the viewpoint of Information Theory, by {\sc Hirschler and
Woess}~\cite{HiWo}. 

In order to carry the above comparison with the infinite
case a bit further, we need some more details on the geometry of an infinite tree $T$ with root
$o$. We assume that $T$ is locally finite and has no leaves. Each vertex $x \ne o$ has a 
unique \emph{predecessor} $x^-$, its neighbour which is closer to $o$. 
For each $x \in T$ there is the unique \emph{geodesic path} $\pi(o,x) = [o=x_0\,,x_1\,,\dots, x_n=x]$
from $o$ to $x$, where $x_k^- = x_{k-1}$ for $k=1,\dots, n$. In this case, $|x|=n$
is the \emph{length} of $x$.

The boundary at infinity $\partial T$ of $T$ consists of all geodesic rays 
$\xi = [o=x_0\,,x_1\,,x_2\,,\dots]$, where $x_k^- = x_{k-1}$ for $k \ge 1$. For a vertex $x \in T$,
we define the \emph{boundary arc}
$$
\partial_x T = \{ \xi \in \partial T : x \in \xi \}\,.
$$
The collection of all $\partial_x T\,$, $x \in T$, is the basis of a topology on $\partial T$,
which thus becomes a compact, totally disconnected space, and each boundary arc is open and compact. 
We now take a Borel probability measure $\Prob$ on $\partial T$ which is supported
by the entire boundary, that is, $\Prob(\partial_x T) > 0$ for all $x \in T$. 
It induces a \emph{forward only} Markov operator on $T$, as follows.
\begin{equation}\label{eq:forw}
p(x,y) = \begin{cases} \Prob(\partial_y T)/\Prob(\partial_x T)\,,&\text{if }\; y^- = x\,,\\
                       0\,,&\text{otherwise.}
         \end{cases}
\end{equation}
Conversely, if we start with transition probabilities $p(x,y)$ such that $p(x,y) > 0$
precisely when $y^- = x$, then we can construct $\Prob$ on $\partial_x T$
by setting 
$$
\Prob(\partial_x T) = p(o,x_1)p(x_1,x_2) \dots p(x_{n-1},x)\,,\quad\text{it}\quad
\pi(o,x) = [o=x_0\,,x_1\,,\dots, x_n=x].
$$
This determines $\Prob$ on the Borel $\sigma$-algebra of $\partial T$.  

More generally, a \emph{distribution} on $\partial T$ is a set function
\begin{equation}\label{eq:distr}
\nu: \{ \partial_x T : x \in T \} \to \C \quad \text{with} \quad
\nu(\partial_x T) =\sum_{y: y^-=x} \nu(\partial_y T)\quad \text{for all }\; x \in T\,.  
\end{equation}
If $\nu$ is non-negative real, then it extends uniquely to a Borel measure on $\partial T$.
A \emph{locally constant function} $\varphi$ on $\partial T$ is one such that every 
$\xi \in \partial T$ has a neighbourhood on which $\varphi$ is constant. Thus, one can write
it as a finite linear combination of boundary arcs
$$
\varphi = \sum_{j=1}^m c_j\,\uno_{\partial_{x(j)} T} \,,
$$
and we can define 
$$
\int_{\partial T} \varphi\, d\nu = \sum_{j=1}^{m} c_j \,\nu(\partial_{x(j)} T)\,.
$$
Indeed, in this way, the space of all distributions is the dual of the linear space of
all locally constant functions on $\partial T$.

Now take $\la \in \C \setminus \{0\}$. Following \cite[\S 6]{PW}, the $\la$-Martin kernel
on $T$ is
$$
K^{(T)}(x,\xi|\la) = \begin{cases} \la^{|x|}/\Prob(\partial_x T)\,,
                                        &\text{if }\;\xi \in \partial_x T\,,\\
                                   0\,, &\text{otherwise.}
                     \end{cases}
$$
For fixed $x$, the function $\xi \mapsto  K^{(T)}(x,\xi|\la)$
and its derivatives with respect to $\la$ are locally constant, whence they
can be integrated against distributions on $\partial T$.
According to \eqref{eq:polyT}, we get 
\begin{equation}\label{eq:KTr}
K^{(T)}_r(x,\xi|\la) = \begin{cases} (-1)^{r-1}\,\la^{|x|-(r-1)}{\displaystyle {|x| \choose r-1}} 
                                    \dfrac{1}{\Prob(\partial_x T)}
                                    \,,&\text{if }\;\xi \in \partial_x T\,,\\[3pt]
                                   0\,, &\text{otherwise,}
                       \end{cases}
\end{equation}
and every $\la$-polyharmonic function of order $n$ on $T$ has a unique representation
\begin{equation}\label{eq:polTfw}
 f(x) = \sum_{r=1}^n (-1)^{r-1}\,\la^{|x|-(r-1)}{|x| \choose r-1} 
\frac{\nu_r(\partial_x T)}{\Prob(\partial_x T)}\,,
\end{equation}
where the $\nu_r = \nu_r^{(T)}$ ($r=1,\dots,n$) are distributions on $\partial T$. 

\smallskip

We now consider the finite situation.
The graph $X$ under consideration is a finite subtree of $T$ with the same root $o$. 
The boundary consists of the \emph{leaves} of the tree:
$$
\partial X = \{ w \in X : w \ne o\,,\; \deg(w) = 1\}\,. 
$$
We suppose that $\partial X$ is a \emph{section} of $T$ in the sense of \cite{HiWo}: 
For every $\xi \in \partial T$, the geodesic ray starting from $o$ that represents $\xi$
intersects $\partial X$ in a unique vertex. (A typical special case is the one where
$\partial X = \{ x \in T : |x|= L\}$ with $L \in \N$.) 
For each $x \in X$, we define the finite version of 
the boundary arc rooted at $x$ as
$$
\partial_x X = \{ w \in \partial X : x \in \pi(o,w) \}.
$$
In particular, $\partial_o X = \partial X$, and $\partial_w X = \{w\}$ for $w \in \partial X$. 

We consider the restriction to $X$ of the given forward transistion matrix $P_T$ on
$T$. That~is, 
$$
p_X(x,y) = \Prob(\partial_y T)/\Prob(\partial_x T)\,,\; \text{ if }\; y^- = x \in X^o,
\AND p_X(w,w) = 1 \; \text{ if }\; w \in \partial X\,,
$$
while $p_X(x,y)=0$ in all other cases. Exactly as on the whole tree, we have for $x, y \in X$ 
$$
\begin{aligned}
&p^{(n)}(x,y) > 0 \iff x \in \pi(o,y) \; \text{and} \; n = |y|-|x| \,,\\
&\text{and then} \quad
 p^{(n)}(x,y) = \Prob(\partial_y T)/\Prob(\partial_x T)\,.
\end{aligned}
$$
The matrix $P_{X^o}$ is nilpotent, so that $\spec(P) = \{ 0, 1\}$, and the algebraic multiplicities
of those two eigenvalues are $|X^o|$ and $|\partial X|$, respectively.
For $\la \in \C \setminus \{0\} = \res(P_{X^o})$ and $x, y \in X^o$, we have
$$
G(x,y|\la) = \begin{cases} \la^{-d(x,y)-1}\, \Prob(\partial_y T)/\Prob(\partial_x T)\,,
                                            &\text{if }\; x \in \pi(o,y)\,,\\
                           0 \,, &\text{otherwise.}
             \end{cases}
$$
Therefore in this example, the right hand side of \eqref{eq:deriveG} is obtained by
$$
\frac{(-1)^{r-1}}{(r-1)!}\, G^{(r-1)}(x,y|\la) 
= \la^{-d(x,y)-r}\, {d(x,y) + r-1 \choose r-1} \, \Prob(\partial_y T)/\Prob(\partial_x T)\,,
$$
if $x \in \pi(o,y)\,$. We note that $\res^*(P_{X^o}) = \res(P_{X^o})$ and that
$F(o,w|\la) = \la^{-|w|}\, \Prob(\partial_w T)$ for $w \in \partial X$. We can 
now compute the kernels $K^{(X)}_r$ of \eqref{eq:RiqK} as follows:
\begin{equation}\label{eq:Kr}
K^{(X)}_r(x,w|\la) = \begin{cases} {\displaystyle \la^{|x|-r+1} \, {d(x,w) + r-2 \choose r-1}\, 
                                          \frac{1}{\Prob(\partial_x T)}}\,,
                                         &\text{if }\; w \in \partial_x T\,,\\
                         0\,,&\text{otherwise.}
\end{cases}
\end{equation}
Then, given boundary functions $g_1\,,\dots, g_n\,$, the associated solution of the Riquier problem
reads
$$
f(x) = \sum_{r=1}^n \int_{\partial X} K^{(X)}_r(x,\cdot\,|\la) \,d\nu^{(X)}_r\,,\quad \text{with}\quad 
\nu^{(X)}_r(w) = \la^{-|w|} \,g_r(w)\, \Prob(\partial_w T)\,.
$$
Now consider \eqref{eq:compare} and the fact that $P_{X^o}$ is the restriction of $P_T$ to $X^o$.
In spite of this, when $n \ge 2$
we see that for $w \in \partial X$, the function $x \mapsto K_n^{(X)}(x,w|\la)$ is 
\emph{not} the restriction to $X^o$ of $x \mapsto K_n^{(T)}(x,\xi|\la)$, where $\xi \in \partial_w X$. 
(The value is the same for every such $\xi$, when $x \in X^o$.)
For a closer look, fix $\xi \in \partial_w T$ and let $f(x) = K_n^{(T)}(x,\xi|\la)$ for $x \in X$.
This function solves the Riquier problem on $X$ with boundary functions 
$$
g_r(v) = K^{(T)}_{n+1-r}(w,\xi|\la)\,\delta_w(v)\,,\quad v \in \partial X\,,
$$
or, equivalently, with boundary measures on $\partial X$ 
$$
\nu^{(X)}_r = (-\la)^{n-r} {|w| \choose n-r}\delta_w\,.
$$
Indeed, verification of
$$
K^{(T)}_n(x,\xi|\la) = \sum_{r=1}^n \int_{\partial X} K^{(X)}_r(x,\cdot\,|\la)\, d\nu^{(X)}_r
$$
leads to known combinatorial identity
$$
{|w| \choose n-1} = \sum_{r=1}^n (-1)^{n-r} {|w|-|x|-r-2 \choose r-1}{|w| \choose n-r},
$$
in which $|w|$ and $|x|$ can be arbitrary integers with $|w| > |x| \ge 0$.

\end{document}